\documentclass{amsart}

\usepackage{amsmath,amsfonts,amssymb,amscd}
\usepackage{graphicx, tikz}
\usetikzlibrary{arrows,shapes} 
\usepackage[all]{xy}


\newcommand{\ep}{{\varepsilon}}
\newcommand{\Ll}{{\wedge}}
\newcommand{\Cc}{{\mathbb{C}}}

\newcommand{\Zz}{{\mathbb{Z}}}
\newcommand{\Qq}{{\mathbb{Q}}}

\theoremstyle{plain}
\newtheorem{thm}{Theorem}[section]
\newtheorem{lem}[thm]{Lemma}

\theoremstyle{definition}
\newtheorem{defn}{Definition}
\newtheorem{remark}{Remark}

\begin{document}

\title{Alexander representation of tangles}

\author{Stephen Bigelow}
\address{Department of Mathematics\\University of California at Santa Barbara\\
Santa Barbara, CA 93106\\USA}
\email{bigelow@math.ucsb.edu}
\urladdr{http://www.math.ucsb.edu/~bigelow/}

\author{Alessia Cattabriga}
\address{Dipartimento di Matematica\\Piazza di Porta S. Donato 5, 40126\\Bologna, Italia}
\email{alessia.cattabriga@unibo.it}
\urladdr{http://www.dm.unibo.it/~cattabri/}

\author{Vincent Florens}
\address{Laboratoire de Math\'{e}matiques et de leurs Applications\\Pau UMR CNRS 5142\\France}
\email{vincent.florens@univ-pau.fr}
\urladdr{http://riemann.unizar.es/geotop/pub/vincent/}

\begin{abstract}
A tangle is an oriented $1$-submanifold of the cylinder
whose endpoints lie on the two disks in the boundary of the cylinder.
Using an algebraic tool developed by Lescop, we extend the Burau representation
 of braids to a functor from the category of oriented tangles to the category of $\Zz[t,t^{-1}]$-modules.
For $(1,1)$-tangles (i.e., tangles with one endpoint on each disk)
this invariant coincides with the Alexander polynomial of the link
obtained by taking the closure of the tangle. 
We use the notion of plat position of a tangle
to give a constructive proof of invariance in this case.
\end{abstract}

\maketitle

\section{Introduction}

The Burau representation was the first non trivial representation of
the braid groups. 
It was defined by writing an explicit matrix in $GL_n(\Zz[t,t^{-1}])$ for every generator of $B_n$,
and verifying that these matrices satisfy the braid relations.
Later, it was reinterpreted with the point of view of mapping class
group and action on {\em Burau modules}, which are
homology modules of the punctured disk with twisted coefficients in $\Zz[t,t^{-1}]$, as a particular case of the Magnus representation \cite{birman}.
This representation was also the first endowed with a unitary structure, and extensively studied for its (non)-faithfulness.

Several extensions of the Burau representation have been considered.
The first was due to Le Dimet \cite{ledimet} for string-links, and studied
in more detail by Kirk and Livingston \cite{klw}. 
An extension for the tangle category was defined by
Cimasoni and Turaev \cite{cimtur1,cimtur2};
it takes value in the Lagrangian category,
which contains that of $\Zz[t,t^{-1}]$-modules. 
 
An important aspect of braid group representations
is their connection with polynomial invariants of knots.
One reconstructs the Alexander polynomial of a closed braid
from the Burau representation as $\det ( id - \rho)$.
The other works on the extension of the representation
to string links or tangles
contain some relations with the Alexander polynomial,
but no direct formula.
 
In this paper we use an algebraic tool introduced by Lescop for link
complements to construct an isotopy invariant for oriented tangles. 
This invariant is related to the Alexander module of the tangle
exterior in the cylinder. This approach is
functorial: the invariant provides a
functor from the category of oriented tangles to the category of
$\Zz[t,t^{-1}]$-modules and homomorphisms up to multiplication by a
unit and for $(1,1)$-tangles, it coincides with the Alexander polynomial of the
closure of the tangle. Our construction is a particular case of a more general construction 
developed in \cite{FM}. 

Section 2 is devoted to the construction of the invariant.
In Section 3 we compute
the value of the invariant for the generators of the category. 
In Section 4  the invariant is computed explicitly
when the tangle is written as the plat closure of an oriented braid. This illustrates the similarities with the
approach of Lawrence \cite{law} for the Jones polynomial and with the
 model of Bigelow for a link presented as a plat closure. 
In the last Section we prove the functoriality of the invariant.

It should be noted that the homology of symmetric spaces of the
punctured disk (or configuration spaces), with some twisted
coefficients coincides with the exterior algebra of the Burau modules.
The modules can also be endowed with a Hermitian structure coming from
the exterior power of the twisted intersection form.
In particular, for $(1,1)$-tangles, the calculation could be interpreted as an
intersection (with twisted coefficients) of a set of curves in the
$2$-disk.
  
It is worth mentioning that the model of Bigelow for the Jones
polynomial inspired some construction of categorification. Our
approach here should similarly give rise to a
categorification of the Alexander polynomial --- presumably some form
of Heegaard Floer homology.
 
%===============================================================
\section{The main construction}
\label{definition}

\subsection{The Alexander functions of a module}
\label{functions}

Let $R= \Zz[t,t^{-1}]$ be the Laurent polynomial ring. Let $H$ be an $R$-module of finite type. 
Consider a free resolution of $H$ of the form 
\begin{eqnarray}
\label{prese}
R^{p}\stackrel{A} \longrightarrow R^{p+k} \longrightarrow H \longrightarrow 0.
\end{eqnarray}
In other words, (\ref{prese}) determines  a $k$-deficiency presentation of $H$ over $R$ with $(p+k)$ generators $\gamma_1,\dots,\gamma_{p+k}$ and $p$ relators
$\rho_1,\dots,\rho_p$. Given such a presentation, it is possible to
define an $R$-linear map  $\Ll^k H \longrightarrow R$ as follows (see \cite[\S 3.1]{le}).

Let $\hat{\rho}=\rho_1\wedge\dots\wedge\rho_{p}$ 
and $\hat{\gamma}=\gamma_1 \wedge \dots \wedge \gamma_{p+k}$. The
{\em Alexander function of $H$ relative to $k$} is the $R$-linear map
$\varphi = \varphi(H,k) \colon \Ll^k H \to R$
defined by
\begin{eqnarray} 
\label{phi}u \wedge \hat{\rho} = \varphi(u) \cdot \hat{\gamma}
\end{eqnarray}
for each $u\in\Ll^k H$.  
For fixed $k$,
different $k$-deficient presentations will give rise to Alexander functions that differ
only by multiplication by a unit in $R$.

Note that if $\hat{u}= u_1 \wedge \dots \wedge u_k$ and $A\in\mathbf{M}_{p\times(p+k)}(R)$ 
is the matrix associated to the presentation (\ref{prese}),  then $\varphi(\hat{u})$ coincides with the determinant of the matrix constructed adding to  $A$  the $k$ columns $u_1,\dots,u_k$ 
expressed in terms of the generators $\gamma_i$.

\begin{remark} \label{rem} 
From (\ref{prese}), by tensoring  with
$\Qq(t)$, we obtain  a presentation of deficiency $k$ for 
$H \otimes \Qq(t)$. 
So, if 
$k < \mathop{\mathrm{rank}} H = \dim_{\Qq(t)} H \otimes \Qq(t)$,
then $\varphi$ is the zero map. 
Moreover, if $H$ is free of rank $k$, then $\varphi$ coincides with the volume form
$$ \Ll^k H \stackrel{\cong} \longrightarrow R $$
induced by the choice of $\hat{\gamma}$ as a basis of $\Ll^k H$,
i.e., $\varphi(\gamma_1 \wedge \dots \wedge \gamma_k)=1$.
\end{remark}

%--------------------------------------------
\subsection{The Alexander  invariant of a tangle}

Let $D$ be the closed unitary disk in $\Cc$. Consider two copies of
the disk $D$ with fixed  ordered finite sequences of points
$p_0,\dots,p_{k_-}$ (respectively $p_0,\dots, p_{k_+}$) on the real line.
Let $\varepsilon^-$ and $\varepsilon^+$ be sequences of signs $\pm 1$
of length $k_-+1$ and $k_++1$, respectively.
An {\em $(\ep^-,\ep^+)$-tangle} is an oriented $1$-submanifold $\tau$
of $D \times I$ whose oriented boundary $\partial\tau$ is
\begin{eqnarray}
\sum_{i=0}^{k_+}\ep_i^+(p_i,1) - \sum_{i=0}^{k_-}\ep_i^-(p_i,0).
\end{eqnarray}
Note that if such a tangle exists then
$\sum_{i=1}^{k_+}\ep_i^+=\sum_{i=1}^{k_-}\ep_i^-$. 

Let $D_-$ and $D_+$ be the punctured disks
$(D \times \{ 0 \}) \setminus \tau$
and $(D \times \{ 1 \}) \setminus \tau$, respectively. Any sequences of signs $\ep^\pm$
determine epimorphisms 
$\chi_{\pm} \colon \pi_1(D_\pm,*) \rightarrow \Zz=\langle t\rangle$ 
sending  a  simple loop $x_i$ turning once
around $p_i$ in the  counterclockwise direction to $t^{\epsilon^\pm_i}$.
Let $H_{\pm}$ be the $R$-module $H_1^{\chi_\pm}(D_\pm; R)$, where the
coefficients are twisted by $\chi_\pm$.
Equivalently, $H_{\pm}$ is the first
homology of the infinite cyclic covering induced by
$\chi_{\pm}$.

\begin{remark} \label{basis}
The disk $D_\pm$ deformation retracts to the wedge of circles $x_0 \dots x_{k_\pm}$.
So the $R$-module $H_\pm$ is  freely  generated by
$u_i^\pm=\tilde{x_i}-\tilde{x}_{i+1}$,  for $i=0,\dots, k_\pm$,
where $\tilde{x_i}$ is a  lift of $x_i$.
More precisely,
$u_i^\pm$ will be a loop around both $p_i$ and $p_{i+1}$ if these punctures have opposite signs,
or a figure eight around them if they have the same sign.
\end{remark}

Let $B_\tau$ be the exterior of $\tau$ in $D \times I$. For each
connected component $\tau_j$ of $\tau$,  let   $m_j$ be a meridian
around $\tau_j$  oriented so that its linking number with $\tau_j$ is
one.

Since  $H_1(B_\tau)=\oplus_{j=1}^{n}\Zz m_j$ (see \cite[\S 3.3]{cimtur1}),
the composition of the Hurewicz homomorphism with the
homomorphism $H_1(B_\tau)\to\Zz \langle t\rangle$ sending $m_i$ to
$t$, gives an epimorphism  $\chi \colon \pi_1(B_\tau)\rightarrow \Zz$
extending both  $\chi_+$ and $\chi_-$.  We set 
$H= H_1^\chi(B_\tau; R)$ and denote with $i_\pm$  the maps from $H_\pm$ to $H$, induced by the inclusion.

Let $\delta k =\frac{k_+ - k_-}{2}$.  Our aim is to construct linear maps
$\rho_{\tau,\star}$, of degree $\delta k$, between the $R$-modules 
$\Ll^{\star} H_-$  and $\Ll^{\star} H_+$.

For any generator $w_+$ of $\Ll^{k_+}H_+ \simeq R w_+$, and $r$ with $0 \leq r \leq k_+$, consider the isomorphism, 
$$ d_r : \Ll^{r}H_+ \longrightarrow \text{Hom}(\Ll^{k_+-r}H_+; R),$$
defined by the formula
$ x \wedge y = \big( d_r(x) (y) \big) w_+$, for any $x\in\Ll^rH_+$ and $y\in\Ll^{k_+-r}H_+$.
 
Let $k = \frac{k_-+k_+}{2}$. We will show in Section \ref{plat} that there exists a free deficiency $k$ resolution for $H$ (see Lemma \ref{heeg}).
For any $i$ with $0 \leq i \leq k$, the choice of generators $\hat{\rho}$ for $\Ll^p R^p$ and $
\hat{\gamma}$ for $\Ll^{k+p} R^{k+p}$ similarly induces 
isomorphisms
$$ \varphi_i: \Ll^i H \longrightarrow \text{Hom} ( \Ll^{k-i} H; R),$$
defined by the formula $x \wedge y \wedge \hat{\rho} = \big( \varphi_i(x)(y) \big) \hat{\gamma}$
in $\Ll^{k+p} R^{k+p}$, for any $x\in\Ll^i H$ and $y\in\Ll^{k-i} H$ (see Section \ref{functions}).
  
\begin{defn}
For all $ i \in \{0,\dots, k \}$, let the homomorphism 
$\rho_{\tau,i}:  \Ll^iH_- \rightarrow \Ll^{i + \delta k} H_+$ be defined by the following composition
$$ \Ll^iH_- \stackrel{\Ll^i i_-} \longrightarrow \Ll^i H \stackrel{\varphi_i} \longrightarrow \text{Hom}( \Ll^{k-i} H; R)
\stackrel{ (\Ll^{k-i} i_+)^*} \longrightarrow \text{Hom}(\Ll^{k-i} H_+; R) \stackrel{d_{i+ \delta k}^{-1}} 
\longrightarrow \Ll^{i + \delta k}  H_+. $$
The map $\rho_\tau = \displaystyle{\oplus_{i} \rho_{\tau,i}}$ is an isotopy invariant of $\tau$, defined up to a global multiplication by a unit of $R$. 
We call it the {\em Alexander invariant} of $\tau$.
\end{defn}
 
 The special case of $(1,1)$-tangles, corresponding to $k_-=k_+=0$, illustrates the relation with the Alexander polynomial.
 
\begin{thm}
Let $\tau$ be a $(1,1)$-tangle and let $\hat{\tau}$ be the link
obtained by closing up together the two free strands of $\tau$. Then
$\rho_{\tau}$ is equal to $m_{\Delta(\hat{\tau})}$ where
$\Delta(\hat{\tau})$ is the  Alexander polynomial of  $\hat{\tau}$
and $m_u \colon \Ll^k H_- \to \Ll^k H_+$
denotes the multiplication by $u$.
\end{thm}

\begin{proof}
Since  $k_-=k_+=0$, it follows that $H_{\pm}= \{0 \}$ and $\Ll^{0} \{0 \}= R$. Hence $\rho_\tau$
 is a homomorphism $R \rightarrow R$ that is the multiplication by the determinant of a square presentation
  matrix of $H$. By definition, this determinant is the Alexander polynomial of $\hat{\tau}$.
\end{proof}

%===============================================================
\section{The tangle category}
\label{tangle}

Let $\mathcal{T}$ denote the oriented tangle category, that is, 
the category whose objects are sequences
$\epsilon=(\epsilon_0,\dots,\epsilon_n)$ of signs $\pm$ attached to
the punctures of a punctured disk, and whose morphisms are oriented
tangles.

Let $\mathcal{M}$ be the
category whose objects are $R$-modules and whose morphisms are classes
of homomorphisms modulo the equivalence relations defined by $f\sim g$
if there exists a unit $u\in R$ such that $f=m_u\circ g$ where $m_u$
denotes the multiplication by u.
There is a functor from $\mathcal{T}$ to $\mathcal{M}$
taking $ \epsilon^-\overset{\tau}{\longrightarrow}\epsilon^+$
to the class of modules homomorphisms 
$H_-(D_{\epsilon^-})\overset{\rho_{\tau}}{\longrightarrow}H_{+}(D_{\epsilon^+})$.
We postpone the  proof of this  statement to last section, and, in the following, 
we compute  $\rho_{\tau}$ on the generators of the tangle category.

%--------------------------------------------
\subsection{Braids}
\label{braid}

Let us start with an oriented identity braid $\tau$ with $k+1$ strands.
Thus $H$ is free of rank $k$,
and $\varphi \colon \Ll^k H \rightarrow R$ 
is an isomorphism and coincides with the volume form
induced by the choice of a basis. Moreover, $i_-$ and $i_+$ are
isomorphisms, and the choice of $\hbox{det}_+$ induces
a canonical choice of  $\varphi$, such that the following diagram
commutes.
   
$$\xymatrix{
  \Ll^k H_+ \ar[rr]^\cong_{i_+} \ar[rd]_{\det_+}
  && \Ll^k H \ar[ld]^\varphi
  \\ & R
}$$

Let $u_-$ be in $\Ll^i H_-$. From the diagram, for all $w_+ \in
\Ll^{k-i} H_+$, we get
  
$$\textstyle{\det}_+ (i_+^{-1} (i_-(u_-) \wedge i_+(w_+)) = \varphi(i_-(u_-) \wedge i_+(w_+)).$$
  By definition,
$$\varphi(i_-(u_-) \wedge i_+(w_+)) = \textstyle{\det}_+ (\rho_\tau(u_-) \wedge w_+ ).$$
  Since $\hbox{det}_+$ is non singular, we get $$\rho_\tau(u_-) = i_+^{-1} \big( i_- (u_-) \big).$$

Now consider the general case of an oriented braid $\sigma$ in $B_{k+1}$,
i.e., $\sigma$ is the isotopy class of a homeomorphism  of  the
$(k+1)$-punctured disk (that we still denote with $\sigma$). 
The geometric realisation of $\sigma$ as a braid with $(k+1)$-strings
in $D \times I$ can be viewed as the mapping cylinder of $\sigma$.
Let $\{u^{\pm}_1,\dots,u^{\pm}_{k_{\pm}}\}$ be the the basis of
$H_\pm$ described in Remark \ref{basis}.
Since $i_- \colon H_-\to H$ is still an isomorphism  we can choose
$\{i_-(u^-_1),\dots, i_-(u^-_{k_{-}})\}$ as a free basis for $H$.
With respect to these bases,  the matrix associated to $i_+$ is
$b(\sigma)^{-1}$, where
$b(\sigma)$ is the image of $\sigma$ by the oriented Burau representation.
It follows that, with respect to these bases the matrix associated to
$(\rho_{\tau})_i \colon \Ll^i H \to\Ll^i H_+$ is 
$\det(b(\sigma))(\Ll^i(b(\sigma)^{-1}))$. 
Moreover, the matrices of the  oriented  Burau representation of the
Artin generators  $\sigma_1,\dots, \sigma_{k-1}$ of $B_{k+1}$  in the
bases $u_1^\pm, \dots, u_{k_\pm}^\pm$  are 
$$ M_{b(\sigma_1)} = 
\begin{pmatrix}
         -t^{\varepsilon_2} & 1 \\
         0 & 1
\end{pmatrix} \oplus Id_{k-3}, 
\ M_{b(\sigma_{k-1})} = Id_{k-3} \oplus \begin{pmatrix}
         1 & 0 \\
         t^{\varepsilon_k} & -t^{\varepsilon_k}
\end{pmatrix} $$
$$ M_{b(\sigma_i)} = Id_{i-2} \oplus  \begin{pmatrix}
        1 & 0 & 0 \\
         t^{\varepsilon_{i+1}} & -t^{\varepsilon_{i+1}} & 1 \\
         0 & 0 & 1
\end{pmatrix} \oplus Id_{k-i-2}. $$

%--------------------------------------------
 \subsection{A cup} 

 Let \emph{cup} be the tangle in Figure \ref{cup}. We have $k_-=k_+-2$,
$k=k_+-1$,  $\delta k=1$ and $H$  free.
Note that $i_-$ is a monomorphism and we can  choose the  bases of
$H_-$ and $H$, such that the matrix
for $i_-$ is $I_{k_-}\oplus 0$. On the other hand, $i_+$ is an
epimorphism  and the kernel of $i_+$ has rank one.
Using Mayer-Vietoris arguments on a small disk around the points
$p_{k+1}$ and $p_{k+2}$ in $D_+$,
we can identify the basis of $H$ with a sub-basis of $H_+$ so that
$i_+$ is simply a projection. 

\begin{figure}[ht]
\begin{center}
\begin{tikzpicture}
%The cylinder:
\draw (0,50pt) ellipse (80pt and 20pt); 
\draw (0,-50pt) ellipse (80pt and 20pt); 
\draw (-80pt,-50pt) -- (-80pt,50pt);
\draw (80pt,-50pt) -- (80pt,50pt);
%The endpoints:
\filldraw (-50pt,50pt) circle (2pt)     node[anchor=south]{$-$};
\filldraw (-30pt,50pt) circle (2pt)     node[anchor=south]{$+$};
\filldraw (-10pt,50pt) circle (2pt)     node[anchor=south]{$+$};
\filldraw (50pt,50pt) circle (2pt)      node[anchor=south]{$+$};
\filldraw (-10pt,-50pt) circle (2pt)     node[anchor=north]{$-$};
\filldraw (50pt,-50pt) circle (2pt)      node[anchor=north]{$-$};
%The tangle:
\draw [->, rounded corners=8pt] (-50pt,50pt) -- (-50pt,17pt)--(-40pt,17pt);
\draw [      rounded corners=8pt] (-50pt,50pt) -- (-50pt,17pt)--(-30pt,17pt)--(-30pt,50pt);
\draw (0,-50pt) ellipse (80pt and 20pt); 
\draw [->] (-10pt,-50pt) -- (-10pt,0pt); \draw (-10pt,-50pt) -- (-10pt,50pt);
\draw [->] (50pt,-50pt) -- (50pt,0pt); \draw (50pt,-50pt) -- (50pt,50pt);
\draw (20pt,0) node{$\dots$};
\end{tikzpicture}
\caption{A cup.}
\label{cup}
\end{center}
\end{figure}

Let $\alpha \in H_+$ be a generator of the kernel of $i_+$. Consider
the {\em contraction along} $\alpha$, that is
$$i_\alpha \colon \hbox{Hom}(\Ll^k H_+, R) \rightarrow
\hbox{Hom}(\Ll^{k_+-1} H_+, R)$$ 
such that
$i_\alpha (\det_+)(v_+) = \det_+(v_+ \wedge \alpha)$,
for any $v_+ \in  \Ll^{k_+-1} H_+$.
In other words, the following diagram commutes. 
$$\xymatrix{
\Ll^{k_+-1} H_+  \ar[rr]^{\wedge \alpha} \ar[rd]_{ i_\alpha (\det_+) }
&& \Ll^{k_+} H_+ \ar[ld]^{\det_+}
\\ & R
}$$

Since $\alpha \in \ker i_+$, one has the following.
$$\xymatrix{
\Ll^{k_+-1} H_+  \ar[rr]^{i_+} \ar[rd]_{i_\alpha (\det_+) }
&& \Ll^{k_+-1} H \ar[ld]^{\varphi}
\\ & R
}$$
   
Let $u_-$ be in $\Ll^i H_-$, for $i= 1, \dots, k$.  By
definition, for all $w_+ \in \Ll^{k-i}H_+$, 
$$ \textstyle{\det}_+ (\rho_\tau(u_-) \wedge w_+)
 = \varphi(i_-(u_-) \wedge i_+(w_+)).$$
By the previous diagrams,
$$ \varphi(i_-(u_-) \wedge i_+(w_+)) 
    = \hbox{det}_+(i_+^{-1} i_-(u_-) \wedge w_+ \wedge \alpha),$$
where by $i_+^{-1} i_-(u_-)$ we mean any preimage of $i_-(u_-)$, since
two of them differ by addition of a multiple of $\alpha$.
From above, since $\det_+$ is non-singular, we get
$$\rho_\tau(u_-) =(-1)^{k-i} i_+^{-1} i_-(u_-) \wedge \alpha .$$

%--------------------------------------------
\subsection{A cap} 

Let {\em cap} be the tangle in Figure \ref{cap}. 
We have $k_-=k_++2$, $k=k_++1$,  $\delta k=-1$ and $H$  free. 
Here $i_+$ is a monomorphism. We write $H = i_+(H_+) \oplus \beta$.
Similarly to the previous section, we can define $i_\beta$ the
contraction along $\beta$ and we have the following commutative
diagrams.
$$
\vcenter{\vbox{
    \xymatrix{
    \Ll^{k_+} H \ar[rr]^{\wedge \beta} \ar[rd]_{i_\beta (\varphi)} 
    && \Ll^{k_+ + 1} H \ar[ld]^{\varphi}
    \\ & R
}
}}
\qquad
\vcenter{\vbox{
    \xymatrix{
    \Ll^{k_+} H_+ \ar[rr]^{i_+} \ar[rd]_{\det_+}
    && \Ll^{k_+} H \ar[ld]^{i_\beta (\varphi) }
    \\ & R
    }
}}$$
\begin{figure}[ht]
\begin{center}
\begin{tikzpicture}
%The cylinder:
\draw (0,50pt) ellipse (80pt and 20pt); 
\draw (0,-50pt) ellipse (80pt and 20pt); 
\draw (-80pt,-50pt) -- (-80pt,50pt);
\draw (80pt,-50pt) -- (80pt,50pt);
%The endpoints:
\filldraw (-50pt,-50pt) circle (2pt)     node[anchor=north]{$+$};
\filldraw (-30pt,-50pt) circle (2pt)     node[anchor=north]{$-$};
\filldraw (-10pt,-50pt) circle (2pt)     node[anchor=north]{$+$};
\filldraw (50pt,-50pt) circle (2pt)      node[anchor=north]{$+$};
\filldraw (-10pt,+50pt) circle (2pt)     node[anchor=south]{$+$};
\filldraw (50pt,+50pt) circle (2pt)      node[anchor=south]{$+$};
%The tangle:
\draw [->, rounded corners=8pt] (-50pt,-50pt) -- (-50pt,-17pt)--(-40pt,-17pt);
\draw [      rounded corners=8pt] (-50pt,-50pt) -- (-50pt,-17pt)--(-30pt,-17pt)--(-30pt,-50pt);
\draw (0,-50pt) ellipse (80pt and 20pt); 
\draw [->] (-10pt,-50pt) -- (-10pt,0pt); \draw (-10pt,-50pt) -- (-10pt,50pt);
\draw [->] (50pt,-50pt) -- (50pt,0pt); \draw (50pt,-50pt) -- (50pt,50pt);
\draw (20pt,0) node{$\dots$};
\end{tikzpicture}
\caption{A cap.}
\label{cap}
\end{center}
\end{figure} 

Let $u_-$ be in $\Ll^{i} H_-$.
For all $w_+ \in \Ll^{k_+ -i+1} H_+$, one has, using the definition and the commutativity of the diagrams
\begin{eqnarray*}
\varphi(i_-(u_-) \wedge i_+(w_+)) &=& \hbox{det}_+ (\rho_\tau(u_-) \wedge w_+) \\
&=& i_\beta (\varphi) \big(i_+(\rho_\tau(u_-)) \wedge i_+(w_+) \big) \\
&=&\varphi \big(i_+(\rho_\tau(u_-)) \wedge   i_+(w_+)\wedge  \beta \big) .
\end{eqnarray*}
So we get
$$i_-(u_-)=(-1)^{k_{+}-i+1} i_+(\rho_\tau(u_-)) \wedge\beta.$$

It follows that $\rho_\tau(u_-)$ is zero unless $u_-$ is of the form
$u_-^\prime \wedge i_-^{-1}(\beta)$,
in which case  $\rho_\tau(u_-)= (-1)^{k_{+}-i+1}i_+^{-1}i_-(u_-^\prime)$.
Here,
$i_-^{-1}(\beta)$ and
$i_+^{-1}i_-(u_-^\prime)$
denote arbitrary elements of the preimage.

%===============================================================
\section{Plat position}
\label{plat}

A system of $n$ disjoint arcs $A_1,\dots, A_n$ properly embedded in
$D\times I$ is called {\em trivial} if they are not linked and are
boundary parallel; more precisely, if there exist $n$ disjoint disks
$D_1,\dots, D_n$ such that $A_i\cap D_i=A_i\subset \partial D_i$,
$\partial D_i \setminus A_i\subset \partial(D\times I)$, and
$A_j\cap D_i=\emptyset$,
for each $i,j=1,\dots,n$ with $i\ne j$.
A {\em trivial tangle of type $(k,n)$} is a tangle whose
component are $k$  vertical strands connecting  $D_+$  with $D_-$ and
a  system of $n$ trivial arcs such that  if there exists an  arc of
$\tau$ connecting two points in $D_+$, then  no arc of $\tau$ connects
two points in $D_-$. 

A {\em Heegaard surface} for a tangle $\tau\subset D\times I$ is
a disk that intersects the tangle transversally and cuts $\tau$ into
two trivial tangles.

To see the existence of a Heegaard surface for a given tangle,
start with a real valued function $f$ on $D\times I$ that is
of Morse type on the  tangle: typically we can imagine $f$ as the height
function on $D\times I$. By perturbing the function, we can
assume that minima occur before maxima; so  if $y=1/2$ is a regular
value separating the minima from the maxima then $f^{-1}(1/2)$ is the
required Heegaard disk.

If we  choose   a standard model $\tau(k,n)$ for a trivial tangle of
type $(k,n)$,  the existence of a Heegaard surface for each tangle can
be rephrased in the following way: each tangle  $(D\times I, \tau)$
can be obtained as
$$(D\times [0,1/2], \tau(k_-,n_-))\cup_{\sigma}
  (D\times [1/2,1], \tau(k_+,n_+)),$$
where $\sigma$ is an automorphism of the $k_-+2n_-=k_++2n_+$
punctured disk $D\times\{1/2\}\cap \tau$. Since everything can be
done up to isotopy, we can think of $\sigma$ as an element of the
braid group $B_{k_-+2n_-}$. We call such a decomposition of a given
tangle a {\em Heegaard splitting} of the tangle.

\begin{lem}
\label{heeg}
Let $\tau$ be a $(\epsilon^-,\epsilon^+)$-tangle where the length of
$\epsilon^{\pm}$ equals to $k_{\pm}+1$. Then any Heegaard splitting of
$\tau$, induces a  presentation of $H$ with deficiency
$k=\frac{k_-+k_+}{2}$.
\end{lem}
 
\begin{proof}
Let
$$(D\times I, \tau)=(D\times [0,1/2], \tau(k_-,n_-))\cup_{\sigma}
(D\times [1/2,1], \tau(k_+,n_+))$$
be a Heegaard splitting for $\tau$,
with $\sigma\in B_{k_-+2n_-}$. 
Set
$$X_1 = \left(D \times \left[0,\frac{1}{2}\right]\right) \setminus \tau(k_-,n_-),$$
$$X_2 = \left(D \times \left[\frac{1}{2},1\right]\right) \setminus \tau(k_+,n_+),$$
and
$$S=\left(D\times\left\{\frac{1}{2}\right\}\right) \setminus \tau(k_-,n_-)
=D \setminus \{p_0,\dots,p_{k_-+2n_{-}}\}.$$   
Thus, $B_\tau= (D \times I) \setminus \tau  =  X_1 \cup_{\sigma} X_2$. By
applying  the Mayer-Vietoris exact sequence to this decomposition we
obtain the exact sequence of reduced homology groups
\begin{eqnarray} \dots \longrightarrow H_1(S) \stackrel{i_{S_-} \oplus i_{S_+}} \longrightarrow H_1(X_1) \oplus H_1(X_2) \longrightarrow H_1(B_{\tau}) \longrightarrow 0, 
\end{eqnarray}
where $i_{S_-}$ is induced by the inclusion $S \hookrightarrow X_1$ 
while  $i_{S_+}$ is induced by $S\stackrel{\sigma}\longrightarrow\sigma(S) \hookrightarrow X_2$.
Let $\chi \colon \pi_1(B_{\tau}) \to \Zz$ be the epimorphism
defining $H=H_1^{\chi}(B_{\tau};R)$. This exact sequence lifts to the
following exact sequence of  homology  with coefficients in $R$.

\begin{eqnarray}
\label{pres}
\dots \longrightarrow H_1^{\chi}(S) \stackrel{\iota_{S_-} \oplus \iota_{S_+}} \longrightarrow H_1^{\chi}(X_1) \oplus H_1^{\chi}(X_2) \longrightarrow H \longrightarrow 0.
\end{eqnarray}
This is a presentation of $H$ as an $R$-module,
since
$H_1^{\chi}(S)$, $H_1^{\chi}(X_1)$ and $H_1^{\chi}(X_2)$ are free $R$-modules.

Since $X_1$, $X_2$ and $S$ have the homotopy type of a wedge of circles, the homology modules have respectively
 ranks $k_-+n_-$, $k_++n_+$ and $k_-+2n_-$. It follows that the deficiency of the presentation is $k$.
\end{proof}

By construction, $\rho_{\tau}$
is well-defined up multiplication by units of $R$.  This implies
that changing the Heegaard splitting of $\tau$  will change
$\rho_{\tau}$ up to multiplication by units of $R$. Nevertheless, we
 give a brief constructive proof of this fact, since it illustrates the similarities
  with the approach of Bigelow for the Jones polynomial \cite{big}. In order to do this we
need to introduce some new definitions.

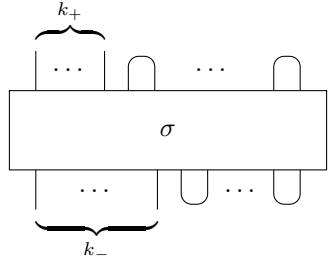
\begin{figure}[ht]
\begin{center}
\begin{tikzpicture}
%The braid:
\draw (-60pt,-15pt) rectangle (60pt,15pt);
\draw (0,0) node{$\sigma$};
%Strands on top:
\draw (-50pt,15pt) -- (-50pt,30pt);
\draw (-37pt,23pt) node{$\dots$};
\draw (-24pt,15pt) -- (-24pt,30pt);
\draw (-37pt,30pt) node[anchor=south]{$\overbrace{\hspace{26pt}}^{k_+}$};
\draw [rounded corners=4pt] (-15pt,15pt) -- (-15pt,28pt)--(-5pt,28pt)--(-5pt,15pt);
\draw (17pt,23pt) node{$\dots$};
\draw [rounded corners=4pt] (40pt,15pt) -- (40pt,28pt)--(50pt,28pt)--(50pt,15pt);
%Strands on bottom:
\draw (-50pt,-15pt) -- (-50pt,-30pt);
\draw (-27pt,-23pt) node{$\dots$};
\draw (-4pt,-15pt) -- (-4pt,-30pt);
\draw (-27pt,-30pt) node[anchor=north]{$\underbrace{\hspace{46pt}}_{k_-}$};
\draw [rounded corners=4pt] (5pt,-15pt) -- (5pt,-28pt)--(15pt,-28pt)--(15pt,-15pt);
\draw (28pt,-23pt) node{$\dots$};
\draw [rounded corners=4pt] (40pt,-15pt) -- (40pt,-28pt)--(50pt,-28pt)--(50pt,-15pt);
\end{tikzpicture}
\caption{The $(k_+,k_-)$-plat closure of $\sigma$.}
\label{pllat}
\end{center}
\end{figure}

Given an oriented braid $\sigma\in B_{r}$, 
and integers $k_+,k_- \leq r$ of the same parity as $r$,
the \emph{$(k_+,k_-)$-plat closure} of $\sigma$ is the tangle  obtained
from $\sigma$ as follows.
Place caps to connect adjacent pairs of endpoints at the top right,
so as to leave only the $k_+$ endpoints at the top left.
Similarly,
place cups to connect adjacent pairs of endpoints at the bottom right,
so as to leave only the $k_-$ endpoints at the bottom left (see Figure \ref{plat}).
Let $\hat{\sigma}(k_+,k_-)$
denote  the $(k_+,k_-)$-plat closure of $\sigma$. Clearly 
$$(D\times
I,\hat{\sigma}(k_+,k_-))=(D\times [0,1/2],
\tau(k_-,n_-) \cup_{\sigma} (D\times [1/2,1], \tau(k_+,n_+)).$$

In order to relate  two braids having the same  $(k_+,k_-)$-plat
closure we define particular subgroups of the braid groups: let
$\mathop{\mathrm{Hil}}_{k+2n}$
be the subgroup of $B_{k+2n}$ generated by the following braids
\begin{itemize}
 \item $\sigma_{k+1}$
\item $\sigma_{k+2}\sigma_{k+1}^2\sigma_{k+2}$
\item $\sigma_{k+2i}\sigma_{k+2i-1}\sigma_{k+2i+1}\sigma_{k+2i}$
for $i=1,\dots n-1$,
\end{itemize}
where $\sigma_1,\dots \sigma_{k+2n-1}$ are the standard generators of $B_{k+2n}$.
When $k=0$ the above subgroups are well known and are called Hilden braid groups on $2n$ strings  (see \cite{BC}).  

We have the following theorem that extend to our setting a result of Birman for classical plat closure of braids  (see  \cite{B2}) . 

\begin{thm}
\label{equivalence}
Two oriented braids have isotopic $(k_+,k_-)$-plat closures 
if and only if they are related by a finite sequence the following moves:  
\begin{itemize}
\item $\sigma\to \sigma k$, 
          where $\sigma\in B_{k_-+2n}$ and  $k\in\mathop{\mathrm{Hil}}_{k_-+2n}$
\item $\sigma\to h \sigma $,
         where $\sigma\in B_{k_++2n}$ and  $h\in\mathop{\mathrm{Hil}}_{k_++2n}$
\item  $\sigma\leftrightarrow \sigma \sigma_{s} $,
         where $\sigma\in B_s$ and  $\sigma_{s}\sigma \in B_{s+2}$.
\end{itemize}
\end{thm}

\begin{proof}
The  proofs of \cite[Theorem 1 and $1'$]{B2} extend to our case. 
\end{proof}

Now we are ready to state our result.

 \begin{thm}
Let $\mathcal H$ and $\mathcal H'$ be  two  Heegaard splittings of the
same $(\epsilon^-,\epsilon^+)$ tangle $\tau$.
Specifically, suppose $\mathcal H$ and $\mathcal H'$ are respectively
$$(D\times I, \tau)=(D\times [0,1/2],\tau(k_-,n_-))  \cup_{\sigma} (D\times [1/2,1], \tau(k_+,n_+)),$$
$$(D\times I, \tau)=(D\times [0,1/2], \tau(k_-,m_-))  \cup_{\varsigma} (D\times [1/2,1], \tau(k_+,m_+)).$$
Moreover, let $\rho_{\tau}$ and $\rho'_{\tau}$  be the module homomorphisms computed  as defined in Section \ref{definition} using the presentations of $H$ associated to the two  Heegaard spittings (see Lemma \ref{heeg}). Then $\rho'_{\tau}=m_u\circ \rho_{\tau}$, where $m_u$ denotes the multiplication by a unit of $R$.
\end{thm}

\begin{proof}

We keep the notations of the proof of Lemma \ref{heeg}.

\begin{figure}[ht]
\begin{center}
\begin{tikzpicture}
\draw (0,0) ellipse (140pt and 55pt); 
\filldraw
               (-110pt,0) circle (2pt) 
               (-90pt,0) circle (2pt)
               (-70pt,0) node {\dots}
               (-50pt,0) circle (2pt)
               (-30pt,0) circle (2pt)
               (-10pt,0) circle (2pt)
               (10pt,0) circle (2pt)
               (30pt,0) circle (2pt)
               (50pt,0) circle (2pt)
               (70pt,0) node {\dots}
               (90pt,0) circle (2pt)
               (110pt,0) circle (2pt);
\draw (-110pt,0) -- (-110pt,-60pt);
\draw (-90pt,0) -- (-90pt,-60pt);
\draw (-50pt,0) -- (-50pt,-60pt);
\draw (-30pt,0) -- (-30pt,-60pt);
\draw [rounded corners=8pt] (-10pt,0) -- (-10pt,-20pt)--(10pt,-20pt)--(10pt,0);
\draw [rounded corners=8pt] (30pt,0) -- (30pt,-20pt)--(50pt,-20pt)--(50pt,0);
\draw [rounded corners=8pt] (90pt,0) -- (90pt,-20pt)--(110pt,-20pt)--(110pt,0);
\draw (-100pt,0) ellipse (16pt and 7pt); \draw (-100pt,10pt) node[anchor=south] {$\gamma_1^-$};
\draw (-40pt,0) ellipse (16pt and 7pt); \draw (-40pt,10pt) node[anchor=south] {$\gamma_k^-$};
\draw (-20pt,0) ellipse (16pt and 7pt); \draw (-20pt,10pt) node[anchor=south] {$\beta_1^-$};
\draw (0pt,0) ellipse (16pt and 7pt); \draw (0pt,10pt) node[anchor=south] {$\alpha_1^-$};
\draw (20pt,0) ellipse (16pt and 7pt); \draw (20pt,10pt) node[anchor=south] {$\beta_2^-$};
\draw (40pt,0) ellipse (16pt and 7pt); \draw (40pt,10pt) node[anchor=south] {$\alpha_2^-$};
\draw (100pt,0) ellipse (16pt and 7pt); \draw (100pt,10pt) node[anchor=south] {$\alpha_{n_-}^-$};
\end{tikzpicture}
\caption{Curves on $S$, and the bottom half of the tangle.}
\label{figx_1}
\end{center}
\end{figure}
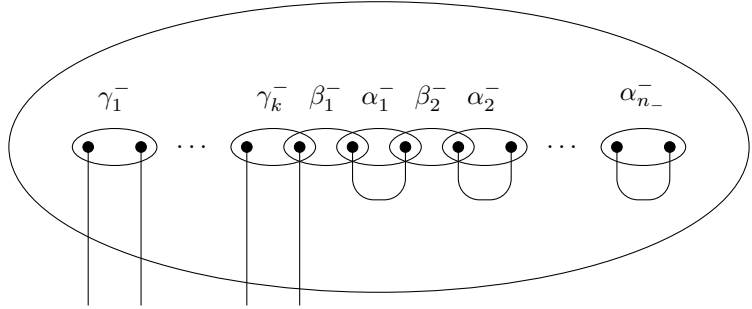

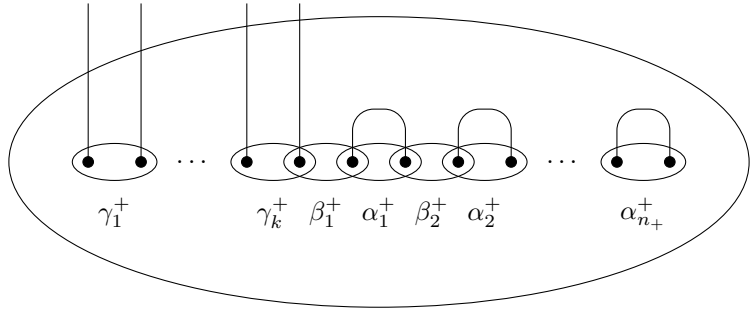
\begin{figure}[ht]
\begin{center}
\begin{tikzpicture}
\draw (0,0) ellipse (140pt and 55pt); 
\filldraw
               (-110pt,0) circle (2pt) 
               (-90pt,0) circle (2pt)
               (-70pt,0) node {\dots}
               (-50pt,0) circle (2pt)
               (-30pt,0) circle (2pt)
               (-10pt,0) circle (2pt)
               (10pt,0) circle (2pt)
               (30pt,0) circle (2pt)
               (50pt,0) circle (2pt)
               (70pt,0) node {\dots}
               (90pt,0) circle (2pt)
               (110pt,0) circle (2pt);
\draw (-110pt,0) -- (-110pt,60pt);
\draw (-90pt,0) -- (-90pt,60pt);
\draw (-50pt,0) -- (-50pt,60pt);
\draw (-30pt,0) -- (-30pt,60pt);
\draw [rounded corners=8pt] (-10pt,0) -- (-10pt,20pt)--(10pt,20pt)--(10pt,0);
\draw [rounded corners=8pt] (30pt,0) -- (30pt,20pt)--(50pt,20pt)--(50pt,0);
\draw [rounded corners=8pt] (90pt,0) -- (90pt,20pt)--(110pt,20pt)--(110pt,0);
\draw (-100pt,0) ellipse (16pt and 7pt); \draw (-100pt,-10pt) node[anchor=north] {$\gamma_1^+$};
\draw (-40pt,0) ellipse (16pt and 7pt); \draw (-40pt,-10pt) node[anchor=north] {$\gamma_k^+$};
\draw (-20pt,0) ellipse (16pt and 7pt); \draw (-20pt,-10pt) node[anchor=north] {$\beta_1^+$};
\draw (0pt,0) ellipse (16pt and 7pt); \draw (0pt,-10pt) node[anchor=north] {$\alpha_1^+$};
\draw (20pt,0) ellipse (16pt and 7pt); \draw (20pt,-10pt) node[anchor=north] {$\beta_2^+$};
\draw (40pt,0) ellipse (16pt and 7pt); \draw (40pt,-10pt) node[anchor=north] {$\alpha_2^+$};
\draw (100pt,0) ellipse (16pt and 7pt); \draw (100pt,-10pt) node[anchor=north] {$\alpha_{n_+}^+$};
\end{tikzpicture}
\caption{Curves on $S$, and the top half of the tangle.}
\label{figx_2}
\end{center}
\end{figure}

Let
$\gamma_1^-,\dots,\gamma_{k_-}^-,  
 \beta_1^-, \alpha_1^-, \dots, \beta_{n_-}^-, \alpha_{n_-}^-$
be the basis for $H_1^{\chi}(S)$
as described in Remark \ref{basis}.
Thus,
each of these is either a loop or figure eight around two adjacent punctures in $S$.
The $\gamma_i^-$ involve two punctures connected to strands that connect to the bottom of the tangle.
The $\alpha_i^-$ are loops around the two endpoints of a cup,
and become zero in the homology of the tangle complement since they could slide off the cup.
The $\beta_i^-$ involve punctures from adjacent cups,
or in the case of $\beta_1^-$,
a cup and a strand.
See Figure \ref{figx_1},
but there we assumed the punctures alternate in sign
so as to avoid having to draw any figure eight.

Now $\gamma_1^-,\dots,\gamma_{k_-}^-, \beta_1^-, \dots, \beta_{n_-}^-$
is a basis for the free $R$-module $H_1^{\chi}(X_1)$,
and $\iota_{S_-}$ is a projection with kernel generated by the $\alpha_i^-$.
Moreover,
$\gamma_1^-,\dots,\gamma_{k_-}^-$ is a basis for the free $R$-module $H_-=H_1^{\chi_-}(D_- ;R)$.

Let 
$\gamma_1^+,\dots,\gamma^+_{k_+},
 \beta_1^+, \dots, \beta_{n_+}^+,
  \alpha_1^+,\dots, \alpha_{n_+}^+$
be the lifts of the loops of  Figure \ref{figx_2}. 
Then $\gamma_1^+,\dots,\gamma^+_{k_+}, \beta_1^+, \dots, \beta_{n_+}^+$ is a
basis for the free $R$-module $H_1^{\chi}(X_2)$. 
Note that $k_-+2n_-=k_+ + 2n_+$.

Let  $(f_1,\dots, f_{k_+)}$ be a ordered basis of $H_+$ as a free $R$-module and let
$$\textstyle{\det}_+ \colon \Ll^{k_+}H_+\to R$$
be the corresponding determinant. For each sequence $I$:
$$1\leq i_1<\dots<i_n\leq k_+-1$$
we set  $\hat{f}_I=f_{i_1}\wedge\dots\wedge f_{i_n}$.
Moreover, we denote with  $\overline I$ the  sequence complementary to
$I$ with respect to $1<2<\dots<k_+$. If $u_-\in\Ll^i H_-$, then $\rho_{\tau}(u_-)=\sum_I
a_I\hat f_I\in \Ll^{i +  \delta k   } H_+ $, where
$$a_I=\textstyle{\det}_+^{-1}\left(\varphi(H,k)(i_-(u_-)\wedge
i_+(\hat{f}_{\overline I}))\right)$$
and $\varphi(H,k)$ is the
Alexander function associated to the presentation of $H$ induced by
$\mathcal H$ as in proof of Lemma \ref{heeg}.

Let $b(\sigma)$ denote the Burau matrix of $\sigma$,
using the above bases.
The presentation matrix of $H$ with respect
to the basis $( \gamma^-_i,\beta^-_j,\alpha^-_j)$ of $H_1^{\chi}(S)$
and  $(\gamma^-_i,\beta^-_j,\gamma^+_h,\beta^+_l)$ of 
$H_1^{\chi}(X_1) \oplus H_1^{\chi}(X_2)$ is
$$
\begin{pmatrix}
    I_{k_-} & 0  &  0 \\
    0 & I_{n_-} & 0 \\
    b(\sigma)(\gamma^-)_{|\gamma^+}
       & b(\sigma)(\beta^-)_{|\gamma^+}
          & b(\sigma)(\alpha^-)_{|\gamma^+} \\
    b(\sigma)(\gamma^-)_{|\beta^+}
       & b(\sigma)(\beta^-)_{|\beta^+}
          & b(\sigma)(\alpha^-)_{|\beta^+}
\end{pmatrix},
$$ 
where the terms of the matrix are blocks.

As a basis of $H_-$ (respectively $H_+$) we can choose the lifts of the loops
$(\gamma_1^-,\dots,\gamma_{k_+}^-)$ (respectively
$(\gamma_1^+,\dots,\gamma_{k_+}^+)$) depicted in Figure \ref{figx_1} and \ref{figx_2}, so
we just have to check how the matrix used to compute
$\varphi(H,k)(i_-(\hat{\gamma}^-_I)\wedge i_+(\hat{\gamma}^+_J))$
changes under the moves of Theorem \ref{equivalence}.
Following the definition of Alexander function and the computation of Lemma \ref{heeg}, we have that
$\varphi(H,k)(i_-(\hat{\gamma}^-_I)\wedge i_+(\hat{\gamma}^+_J))$
is the determinant of the following matrix
\[
M = \begin{pmatrix}
I_{k_-} & 0  &  0 & E^-_I & 0\\
0 & I_{n_-} & 0 & 0 & 0\\
b(\sigma)(\gamma^-)_{|\gamma^+}  & b(\sigma)(\beta^-)_{|\gamma^+} & b(\sigma)(\alpha^-)_{|\gamma^+} &0 & E^+_J\\
b(\sigma)(\gamma^-)_{|\beta^+} & b(\sigma)(\beta^-)_{|\beta^+} & b(\sigma)(\alpha^-)_{|\beta^+} &0 &0 \\
\end{pmatrix},
\] 
where 
\[ (E_I^-)_{hk}=\left\{\begin{array}{l}
                 1\qquad \text{if}\  k=i_h\\
0 \qquad \text{otherwise} 
                \end{array}\right. \qquad \text{and}\qquad (E_J^+)_{hk}=\left\{\begin{array}{l}
                 1\qquad \text{if}\  k=j_h\\
0 \qquad \text{otherwise.} 
                \end{array}\right.
\] 
 
So we want to check how the determinant of the matrix $M$ changes under the moves of Theorem \ref{equivalence}.
 
The last move changes  the braid index by two. Let $N=k_-+2n_-$ and
suppose that $\sigma\in B_{N}$. Then $\varsigma=\sigma\sigma_N\in
B_{N+2}$. Moreover, by the functoriality of the  Burau representation,
$b(\sigma\sigma_N)=b(\sigma)b(\sigma_N)$.
An explicit computation gives the presentation matrix  of $H$,
relatively to the Heegaard decomposition
$\mathcal H'$:
$$
P'=\begin{pmatrix}
        I_{k_-} & 0  &  0 & 0 & 0  \\
        0 & I_{n_-}&0  &0 & 0 \\
0 & 0 & 1&0 & 0 \\       b(\sigma)(\gamma^-)_{|\gamma^+}  & b(\sigma)(\beta^-)_{|\gamma^+} & 0 & b(\sigma)(\alpha^-)_{|\gamma^+} &   \\
        b(\sigma)(\gamma^-)_{|\beta^+} & b(\sigma)(\beta^-)_{|\beta^+} &0 & b(\sigma)(\alpha^-)_{|\beta^+} & 0 \\  0 & 0 &   -t^{\epsilon+1} &0\dots 0\ t^{\epsilon+1} &1 \\
\end{pmatrix},
$$ 
with respect to the bases
$$(\gamma^-_i,\beta^-_j,\beta^-_{n_-+1},\alpha^-_j,\alpha_{n_-+1}),$$
$$(\gamma^-_i,\beta^-_j,\beta^-_{n_-+1},\gamma^+_h,\beta^+_l,\beta^+_{n_++1}).$$
It is straightforward to check that the Alexander function associated
to this presentation matrix is $\pm \det M$.

The first two moves of Theorem \ref{equivalence} do not change the
braid index of $\sigma$ and so the size of $M$. So it is enough to
check how the matrix $M$ changes by composing
$\sigma$  (on the left or on the right) with a generator of
$\mathop{\mathrm{Hil}}_{k+2n}$. We leave the detail of the computations to the reader.
\end{proof} 

%===============================================================
\section{Functoriality}

In this section,
we prove that the Alexander invariant defined in Section \ref{definition} and  computed on  cups, caps and braids,
in Section \ref{tangle}, is a functorial invariant of the tangle category.

\begin{lem}
Let $\tau$ be a tangle in plat position.
The Alexander invariant of $\tau$ from Section \ref{plat}
is the same as
the Alexander invariant computed as 
a product of cups, a braid and caps from Section \ref{tangle}.
\end{lem}

\begin{proof}
Fix a tangle $\tau$,
which consists of a collection of cups at the bottom,
a collection of caps at the top,
and a braid $\sigma$ in the middle.

Let $\gamma^-_i$ be the generators of the homology of the bottom disk of $\tau$,
as shown in Figure \ref{figx_1}.
Fix $\gamma^-_I$, a wedge product of a sequence of these basis vectors.

Let $\gamma^+_j$ be the generators of the homology of the top disk of $\tau$,
as shown in Figure \ref{figx_2}.
Fix $\gamma^-_{\overline J}$,
a wedge product of a sequence of these basis vectors,
where $\overline J$ is the complement of $J$.
We will compute
the coefficient of $\gamma^-_{\overline J}$ in $\rho(\tau)(\gamma^-_I)$.

Let $\alpha^-_i$ be the loops around the cups,
as shown in Figure \ref{figx_1}.
Let $\alpha^-$ be the wedge product of all of these elements in order.
If we apply just the cups from $\tau$ to $\gamma^-_I$,
we obtain $\gamma^-_I \wedge \alpha^-$.
Now apply the braid $\sigma$,
to obtain
$b(\sigma)((\gamma^-_I) \wedge (\alpha^-))$.

Let $\beta^+_i$ be the loops that go between adjacent caps of $\tau$,
as shown in Figure \ref{figx_2}.
Let $\beta^+$ be the wedge product of all of these elements in order.
The coefficient of $\gamma^-_{\overline J}$ in $\rho(\tau)(\gamma^-_I)$
is the coefficient of $\gamma^-_{\overline J} \wedge \beta^+$
in $b(\sigma)(\gamma^-_I) \wedge b(\sigma)(\alpha^-)$.

Now we recompute this same coefficient. Using notation from Section \ref{plat}, 
the coefficient of $\gamma^-_{\overline J}$ in $\rho(\tau)(\gamma^-_I)$
is the determinant of:
\[
\begin{pmatrix}
E^+_{\overline J} & E^+_J
\end{pmatrix}
\] 
times the determinant of:
\[
\begin{pmatrix}
I_{k_-} & 0  &  0 & E^-_I & 0\\
0 & I_{n_-} & 0 & 0 & 0\\
b(\sigma)(\gamma^-)_{|\gamma^+} & b(\sigma)(\beta^-)_{|\gamma^+} & b(\sigma)(\alpha^-)_{|\gamma^+} &0 & E^+_J\\
b(\sigma)(\gamma^-)_{|\beta^+} & b(\sigma)(\beta^-)_{|\beta^+} & b(\sigma)(\alpha^-)_{|\beta^+} &0 &0 \\
\end{pmatrix}.
\] 

The first of these determinants is $\pm 1$,
which we will deal with later.
The second matrix will be simplified using column operations
that do not change the determinant.
Specifically,
use the left column of blocks
to cancel out the $E_I^-$ in the fourth column of blocks.
This gives the following.
\[
\begin{pmatrix}
I_{k_-} & 0  &  0 & 0 & 0\\
0 & I_{n_-} & 0 & 0 & 0\\
b(\sigma)(\gamma^-)_{|\gamma^+} & b(\sigma)(\beta^-)_{|\gamma^+} & b(\sigma)(\alpha^-)_{|\gamma^+}
& -b(\sigma)(\gamma^-_I)_{|\gamma^+} & E^+_J\\
b(\sigma)(\gamma^-)_{|\beta^+} & b(\sigma)(\beta^-)_{|\beta^+} & b(\sigma)(\alpha^-)_{|\beta^+}
& -b(\sigma)(\gamma^-_I)_{|\beta^+} &0 \\
\end{pmatrix}.
\] 

This has the same determinant as the smaller matrix from its bottom right.
\[
\begin{pmatrix}
b(\sigma)(\alpha^-)_{|\gamma^+} & -b(\sigma)(\gamma^-_I)_{|\gamma^+} & E^+_J\\
b(\sigma)(\alpha^-)_{|\beta^+} & -b(\sigma)(\gamma^-_I)_{|\beta^+} &0 \\
\end{pmatrix}.
\] 
Up to sign,
this is the same as the determinant of:
\[
\begin{pmatrix}
b(\sigma)(\gamma^-_I)_{|\gamma^+} & E^+_J & b(\sigma)(\alpha^-)_{|\gamma^+} \\
b(\sigma)(\gamma^-_I)_{|\beta^+} & 0 & b(\sigma)(\alpha^-)_{|\beta^+} \\
\end{pmatrix},
\] 
Here we switched two block columns and changed the sign of one.
This may change the sign of the determinant,
depending on the parity of the number of basis vectors in $\gamma^-_I$.
We can ignore this,
since we are free to consistently change the sign of every entry of the matrix
$\rho(\tau)$.

Finally,
observe that the left and right columns of the above matrix
represent the terms of
$b(\sigma)(\gamma^-_I) \wedge b(\sigma)(\alpha^-)$
that are wedge products of vectors from $\gamma^+$ and $\beta^+$.
The determinant picks out the terms
that involve precisely the vectors
$\gamma^+_{\overline J}$ and $\beta^+$.
Such terms are, up to sign,
$\gamma^+_{\overline J} \wedge \beta^+$.
The sign correction is the determinant of
\[
\begin{pmatrix}
E^+_{\overline J} & E^+_J & 0 \\
0 & 0 & I \\
\end{pmatrix}
\] 
We conclude that
this direct calculation using the methods of Section \ref{plat}
gave the same answer as composing the operations of
cup, braid, and cap.
\end{proof}

\begin{thm}
The Alexander invariant 
as computed  in Section \ref{tangle} on cups braid and caps
determines a functorial invariant of the tangle category.
\end{thm}

\begin{proof}
A presentation of the tangle category
by generators and relations
can be found in \cite{t}.
Other presentations appear in the literature,
sometimes with minor differences.
Certainly cups, caps, and generators of the braid groups
are enough to generate the tangle category.
By the previous lemma,
we have every relation that equates different plat presentations for the same tangle.
These,
together with some simple commutativity relations,
are enough to give defining relations for the tangle category.
\end{proof}

%=====================================

\newcommand{\etalchar}[1]{$^{#1}$}
\providecommand{\bysame}{\leavevmode\hbox to3em{\hrulefill}\thinspace}
\providecommand{\MR}{\relax\ifhmode\unskip\space\fi MR }
\providecommand{\MRhref}[2]{%
  \href{http://www.ams.org/mathscinet-getitem?mr=#1}{#2}
}
\providecommand{\href}[2]{#2}

\end{document}